\documentclass[letterpaper, 10 pt, conference]{IEEEconf}  

\IEEEoverridecommandlockouts                              
\usepackage{cite}
\usepackage{url}
\usepackage{amsmath,amssymb,amsfonts}
\usepackage{multirow}
\usepackage{algorithmic}
\usepackage{graphicx}
\usepackage{textcomp}
\usepackage{epsfig}
\usepackage{algorithmic}
\usepackage{epstopdf}
\usepackage[ruled]{algorithm}
\usepackage{enumerate}
\usepackage{float}
\usepackage{graphics,graphicx}
\usepackage{subfigure}
\usepackage{booktabs}  

\usepackage{array}
\usepackage{mathtools}
\usepackage[justification=centering]{caption}
\allowdisplaybreaks[4]
\usepackage{url}
\DeclareGraphicsExtensions{.png}
\graphicspath{{./Pics/}}
\usepackage{color}
\usepackage{xspace}

\newcommand{\mb}{\mathbf}

\renewcommand{\Re}{{\mathbb{R}}}

\newtheorem{lemma}{Lemma}
\newtheorem{theorem}{Theorem}

\newtheorem{assumption}{Assumption}

\title{\LARGE \bf
A Mixing-Accelerated Primal-Dual Proximal Algorithm for \\Distributed Nonconvex Optimization
}

\author{Zichong Ou, Chenyang Qiu, Dandan Wang and Jie Lu
\thanks{Z. Ou, D. Wang, J. Lu are with the School of Information Science and Technology, Shanghaitech University, 2012 Shanghai, China. J. Lu is also with the Shanghai Engineering Research Center of Energy Efficient and Custon AI IC, 201210 Shanghai, China. Email: {\tt\small ouzch, wangdd2, lujie@shanghaitech.edu.com}.}\\
\thanks{C. Qiu is with Charles L. Brown Department of Electrical and Computer Engineering, University of Virginia, Charlottesville, VA 22904- 4743, USA. Email: {\tt\small nzp4an@virginia.edu}.}%
}

\begin{document}

\maketitle
\thispagestyle{empty}
\pagestyle{empty}

\begin{abstract} 
	In this paper, we develop a distributed mixing-accelerated primal-dual proximal algorithm, referred to as MAP-Pro, which enables nodes in multi-agent networks to cooperatively minimize the sum of their nonconvex, smooth local cost functions in a decentralized fashion. The proposed algorithm is constructed upon minimizing a computationally inexpensive augmented-Lagrangian-like function and incorporating a time-varying mixing polynomial to expedite information fusion across the network. The convergence results derived for MAP-Pro include a sublinear rate of convergence to a stationary solution and, under the Polyak-{\L}ojasiewics (P-{\L}) condition, a linear rate of convergence to the global optimal solution. Additionally, we may embed the well-noted Chebyshev acceleration scheme in MAP-Pro, which generates a specific sequence of mixing polynomials with given degrees and enhances the convergence performance based on MAP-Pro. Finally, we illustrate the competitive convergence speed and communication efficiency of MAP-Pro via a numerical example.
\end{abstract}

\section{INTRODUCTION}

Distributed optimization has gained significant attention in recent years due to its high efficiency, robustness and scalability in numerical analysis and engineering applications \cite{Nedi2015ConvergenceRO}. In the past decade, a large number of distributed convex optimization algorithms has been come up with. For example, \cite{Nedi2009DistributedSM} proposes a distributed subgradient algorithm for solving convex optimization problems and is shown to sublinearly converge to a suboptimal soultion for convex, smooth problem in \cite{Yuan2013OnTC}. In \cite{Shi2014EXTRAAE}, an exact first-order algorithm (EXTRA) is developed, which guarantees linear convergence under the restricted strong convexity condition. In addition, AMM \cite{Wu2020AUA} introduces a primal-dual paradigm for convex composite optimization problems, which unifies a series of recent distributed optimization algorithms including EXTRA. 

However, most distributed convex optimization algorithms cannot be extended to address the more challenging nonconvex optimization problems that arise in a substantial number of real-world applications, such as dictionary learning \cite{Wai2015ACD}, power allocation \cite{Bianchi2011ConvergenceOA}, energy efficiency in mobile ad hoc networks \cite{hashempour2021distributed}, just to name a few. Compared to the abundant literature of distributed convex optimization, so far relatively few algorithms \cite{Alghunaim2021AUA, scutari2013decomposition, sun2016distributed, hong2017prox ,sun2018distributed, sun2019distributed, yi2021linear, yi2022sublinear} have been proposed to address a variety of nonconvex problems based on various techniques. For instance, SUDA \cite{Alghunaim2021AUA} generalizes several decentralized algorithms, including EXTRA, and has shown their effectiveness in nonconvex, smooth optimization. Utilizing the Successive Convex Approximation (SCA) technique \cite{scutari2013decomposition} to partially linearize the nonconvex sum-utility functions, SONATA \cite{sun2016distributed} is shown to asymptotically converge to a stationary solution for nonconvex, nonsmooth optimization. Several primal-dual schemes have been presented in \cite{hong2017prox} and \cite{sun2018distributed} whose primal updates involve minimization of a linearized augmented Lagrangian (AL) function and are guaranteed to sublinearly converge to a stationary solution with nonconvex, smooth cost functions. xFILTER \cite{sun2019distributed} reformulates the nonconvex, smooth problem to a regularized least-squares problem, and employs Chebyshev iterations to approximate the solution at each iteration, which facilitate convergence by increasing communication rounds.

Most of these existing distributed nonconvex optimization algorithms guarantee the convergence to stationarity \cite{Alghunaim2021AUA,scutari2013decomposition,sun2016distributed,hong2017prox,sun2018distributed,sun2019distributed,yi2022sublinear,yi2021linear}, while better convergence results can be obtained under the Polyak-{\L}ojasiewics (P-{\L}) condition, which is a nonconvex setting weaker than strong convexity. For instance, SUDA \cite{Alghunaim2021AUA} converges linearly to the global optimum if the global cost function satisfies the P-{\L} condition. Similar results can be found in \cite{yi2021linear} and \cite{yi2022sublinear}. In particular, based on the primal-dual gradient method, \cite{yi2021linear} conducts gradient descent of AL function, while \cite{yi2022sublinear} modifies the well-known Alternating Direction Method of Multipliers (ADMM), yielding the distributed linearized ADMM (L-ADMM) algorithm.

In this paper, we approach a nonconvex and smooth distributed optimization problem in a primal-dual manner. To reduce computational cost caused by the primal minimization operations such as those in \cite{hong2017prox} and \cite{sun2018distributed}, we design an augmented-Lagrangian-like function consisting of a linearization of the global cost function and a time-varying proximal term, and minimize this function at each primal iteration. To facilitate information mixing throughout the entire network, we let the proximal term involve a polynomial of mixing matrices with time-varying degree and coefficients, which can be executed through an inner loop of multiple communication rounds within each iteration. Thus, the resulting algorithm is called the \textit{Mixing-Accelerated Primal-Dual Proximal algorithm}, referred to as MAP-Pro. To further enhance the convergence speed and communication efficiency, we integrate MAP-Pro with Chebyshev acceleration \cite{auzinger2011iterative} that specializes the mixing polynomials, leading to a special case of MAP-Pro called MAP-Pro-CA.

We show that MAP-Pro converges to a stationary solution at a sublinear rate for nonconvex smooth optimization. MAP-Pro is able to achieve a linear convergence rate to the global optimum, provided that the P-{\L} condition is additionally assumed. With Chebyshev acceleration, superior convergence rates are obtained compared with the existing works in \cite{yi2021linear} and \cite{yi2022sublinear}. Simulations on a distributed binary classification problem with nonconvex regularizers demonstrate that MAP-Pro-CA outperforms a number of state-of-the-art alternative methods in terms of communication efficiency.

The outline of the paper is as follows: Section~\ref{section problem form} formulates the nonconvex optimization problem, and Section~\ref{section algorithm design} describes the development of MAP-Pro. Section~\ref{convergence analysis} states the convergence results for MAP-Pro, and Section~\ref{section discuss} discusses the mixing acceleration scheme. Section~\ref{section simulation} compares the numerical results of MAP-Pro and several alternative methods, and Section~\ref{section conclusion} concludes the paper.

\emph{Notation:} Given any differentiable function $f$, $\nabla f$ denotes the gradient of $f$. $\operatorname{Null}(\cdot)$ and $\operatorname{span}(\cdot)$ denote the null space and range space of the matrix argument. $\mathbf{1}_n$ ($\mathbf{0}_n$) denotes the column one (zero) vector of dimension $n$. $\mathbf{I}_n$, $\mathbf{O}_n$ denote the $n$-dimensional identity matrix and zero matrix, respectively. Besides, for $a_i\in \mathbb{R},i=1,\dots,n$, $\operatorname{diag}([a_1,\cdots,a_n])$ represents the $n$-dimensional diagonal matrix, whose $i$-th diagonal entry is $a_i$. In addition, $\langle \cdot,\cdot \rangle$ is the Euclidean inner product, $\otimes$ is the Kronecker product, and $\|\cdot\|$ is the $\ell_2$ norm. For any two matrices $A,B\in \mathbb{R}^{d\times d}$, $A\succ B$ means $A-B$ is positive definite, and $A \succeq B$ means $A-B$ is positive semi-definite. $\bar{\lambda}_A$ and $\underline{\lambda}_A$ represent the largest and the smallest positive eigenvalues of matrix $A$. Besides, $A^{\dagger}$ denotes the Moore-Penrose inverse of matrix $A$. If $A$ is symmetric and $A \succeq \mathbf{O}_d$, for $\mb{x}\in \mathbb{R}^d$, $\|\mb{x}\|^2_A:=\mb{x}^{\mathsf{T}}A\mb{x}$.

\section{PROBLEM FORMULATION}\label{section problem form}

Consider a network of $N$ nodes, where each node $i$ has a local smooth and possibly nonconvex differentiable cost function $f_i : \mathbb{R}^d \rightarrow \mathbb{R}$. And all the nodes are required to collaboratively solve the following optimization problem:
\begin{equation}
\min _{x \in \mathbb{R}^{d}} f(x)= \sum_{i=1}^{N} f_{i}(x). \label{p1}
\end{equation}

We model the network as a connected, undirected graph $\mathcal{G} = (\mathcal{V},\mathcal{E})$, where the vertex set $\mathcal{V} = \{1, \dots, N\}$ is the set of $N$ nodes and the edge set $\mathcal{E} \subseteq \{\{i,j\}|i,j \in \mathcal{V},i \neq j\}$ describes the underlying interactions among the nodes, i.e., each node only communicates with the neighboring nodes in $\mathcal{N}_i=\{j\in\mathcal{V}:\{i,j\}\in\mathcal{E}\}$. 

Next, we impose the following problem assumptions on problem~\eqref{p1}.
\begin{assumption}\label{assumption smooth}
The global cost function $f:\Re^{d}\rightarrow\Re$ is smooth, i.e., there exists $\bar{M}\geq 0$ such that
\begin{equation}
\|\nabla f(x) - \nabla f(y)\|\leq \bar{M}\|x-y\|,\quad\forall x,y\in \Re^{d}.\label{smooth}
\end{equation}
\end{assumption}

\begin{assumption}\label{a1}
The optimal set $\mathbb{X}^*$ of problem~\eqref{p1} is nonempty and the optimal function value $f^* >-\infty$.
\end{assumption}

Assumptions~\ref{assumption smooth} and~\ref{a1} are commonly adopted in existing distributed nonconvex optimization works \cite{sun2018distributed, sun2019distributed, yi2021linear, yi2022sublinear}. Also, compared with the smoothness of each $f_i$ assumed in \cite{yi2022sublinear,yi2021linear}, Assumption~\ref{assumption smooth} only requires the smoothness of the global cost function ${f}$ and thus is slightly less restrictive.

To solve problem~\eqref{p1} over the graph $\mathcal{G}$, we let each node $i \in \mathcal{V}$ maintain a local estimate $x_i \in \mathbb{R}^d$ of the global decision $x \in \mathbb{R}^d$ in problem~\eqref{p1}, and define 
\begin{equation}
\tilde{f}(\mathbf{x}) := \sum _{i \in \mathcal{V}} f_i(x_i),\quad \mb{x} = (x_1^{\mathsf{T}}, \dots, x_N^{\mathsf{T}})^{\mathsf{T}} \in \Re ^{Nd}.
\end{equation}
It has been shown in \cite{mokhtari2016decentralized} that problem~\eqref{p1} can be equivalently transformed  into 
\begin{equation}
\begin{aligned}
\underset{\mathbf{x} \in \mathbb{R}^{N d}}{\operatorname{minimize}}\quad & \tilde{f}(\mb{x}) \\
\operatorname{subject\;to}\quad & \tilde{H}^{\frac{1}{2}} \mb{x}=0,
\end{aligned} \label{p2}
\end{equation}
where $\tilde{H} = \tilde{H}^{\mathsf{T}}\succeq\mb{O}_{Nd}$ such that $\operatorname{Null}(\tilde{H}) = \mathcal{S} := \{\mb{x} \in \Re ^{Nd}|x_1 = \cdots = x_N\}$. The optimal value of problem~\eqref{p2} is the same as that of \eqref{p1}. Obviously, under Assumption~\ref{assumption smooth}, $\tilde{f}$ is $\bar{M}-$smooth.

\section{ALGORITHM DEVELOPMENT}\label{section algorithm design}

In this section, we develop a distributed algorithm for solving the nonconvex optimization problem~\eqref{p2}, which attempts to accelerate information mixing across the network.

To this end, we first consider solving \eqref{p2} in the following primal-dual fashion. Let $\mathbf{x}^k,\mb{v}^{k}\in \mathbb{R}^{Nd}$ be the global primal and dual variables at iteration $k \geq 0$, respectively. Then, starting from any $\mb{x}^0,\mb{v}^0 \in \mathbb{R}^{Nd}$, for any $k \geq 0$,
\begin{align}
	\label{argmin approximation primal} \notag \mb{x}^{k+1}=&\underset{\mathbf{x} \in \mathbb{R}^{N d}}{\arg \min}\;\tilde{f}(\mb{x}^k) + \langle \nabla \tilde{f}(\mb{x}^k), \mb{x}-\mb{x}^k \rangle +\langle\theta\mb{v}^{k}, \tilde{H}^{\frac{1}{2}} \mb{x}\rangle\\ 
    &\qquad\qquad+\frac{\rho}{2}\|\mb{x}\|_{H}^{2}+ \frac{1}{2} \|\mb{x}-\mb{x}^k\|^2_{B^k}, \\ 
	\label{argmin approximation dual} \mb{v}^{k+1}=&\mb{v}^{k}+\rho \tilde{H}^{\frac{1}{2}} \mb{x}^{k+1}.
\end{align}
In the primal update \eqref{argmin approximation primal}, $\mb{x}^{k+1}$ is obtained by minimizing an augmented-Lagrangian-like function, which consists of 
\begin{itemize}
\item $\tilde{f}(\mb{x}^k) + \langle \nabla \tilde{f}(\mb{x}^k), \mb{x}-\mb{x}^k \rangle+\theta(\mb{v}^{k})^{\mathsf{T}} \tilde{H}^{\frac{1}{2}} \mb{x}$, which approximates the Lagrangian function $\tilde{f}(\mb{x})+(\mb{v}^{k})^{\mathsf{T}} \tilde{H}^{\frac{1}{2}} \mb{x}$ by replacing $\tilde{f}(\mb{x})$ with its first-order approximation at $\mb{x}^k$ and adding an additional weight $\theta>0$.
\item $\frac{\rho}{2}\|\mb{x}\|_{H}^{2}$, $\rho>0$, a penalty on the consensus constraint $\tilde{H}^{\frac{1}{2}} \mb{x}=0$. Note that here we use a different weight matrix $H=H^{\mathsf{T}}\succeq\mb{O}_{Nd}$ other than $\tilde{H}$ in order to introduce more flexibility and, like $\tilde{H}$, require $\operatorname{Null}(H) = \mathcal{S}$ so that it is a legitimate penalty on the consensus residual.
\item $\frac{1}{2} \|\mb{x}-\mb{x}^k\|^2_{B^k}$, a proximal term with a possibly time-varying symmetric matrix $B^k\in \mathbb{R}^{Nd \times Nd}$.
\end{itemize}
We require that $B^k + \rho H \succ \mathbf{O}$, so that $\mb{x}^{k+1}$ in \eqref{argmin approximation primal} is well-defined and uniquely exists. The update \eqref{argmin approximation dual} is a dual ascent step with respect to the augmented-Lagrangian-like function, where the corresponding ``dual gradient" is obtained by evaluating the constraint residual at $\mb{x}^{k+1}$. 

It can be shown that any primal-dual optimum pair $(\mb{x}^*,\mb{v}^*)$ of problem~\eqref{p2} is a fixed point of \eqref{argmin approximation primal}--\eqref{argmin approximation dual}. To see this, note from \eqref{argmin approximation primal} that $\mb{x}^{k+1}$ uniquely exists and satisfies the first-order optimality condition
\begin{equation}
\label{first-order opt} \nabla \tilde{f}(\mb{x}^k) + \theta\tilde{H}^{\frac{1}{2}} \mb{v}^k + \rho H \mb{x}^{k+1}+ B^k(\mb{x}^{k+1} - \mb{x}^k)=0.
\end{equation}
Also note that $(\mb{x}^*,\mb{v}^*)$ satisfies
\begin{equation}
\nabla \tilde{f}(\mb{x}^*)+ \theta\tilde{H}^{\frac{1}{2}} \mb{v}^*=0. \label{fixed point}
\end{equation}
Once $(\mb{x}^k,\mb{v}^k) = (\mb{x}^*,\mb{v}^*)$, $(\mb{x}^{k+1},\mb{v}^{k+1})$ has to be $(\mb{x}^*,\mb{v}^*)$ due to \eqref{fixed point}, $H\mb{x}^*=\tilde{H}^{\frac{1}{2}}\mb{x}^*=0$, and the uniqueness of $\mb{x}^{k+1}$.

Next, we discuss our distributed algorithm design based on \eqref{argmin approximation primal}--\eqref{argmin approximation dual}. Using \eqref{first-order opt}, we rewrite \eqref{argmin approximation primal} as
\begin{equation}
	\mb{x}^{k+1} = \mb{x}^k - (B^k + \rho H)^{-1}(\nabla \tilde{f}(\mb{x}^k) + \theta\tilde{H}^{\frac{1}{2}} \mb{v}^k + \rho H \mb{x}^k). \label{x^k+1 vk}
\end{equation}
Then, we apply the following change of variable
\begin{equation*}
\mb{q}^k = \tilde{H}^{\frac{1}{2}} \mb{v}^k.
\end{equation*}
This requires $\mathbf{q}^k \in \mathcal{S}^{\perp}\, \forall k \geq 0$, where $\mathcal{S}^{\perp} := \{\mb{x}\in \Re ^{Nd}| \, x_1 + \cdots + x_N = \mb{0}\}$ is the orthogonal complement of $\mathcal{S}$. This can be simply guaranteed by
\begin{equation}
\mb{q}^0 \in \mathcal{S}^{\perp}.\label{q0}
\end{equation}
Furthermore, for convenience, let $G^k:= (B^k + \rho H)^{-1}$ and we introduce an auxiliary variable
\begin{equation}
\label{zk original}\mb{z}^{k} = \nabla \tilde{f}(\mb{x}^k) + \theta\mb{q}^k + \rho H \mb{x}^k,
\end{equation}
where $\mb{x}^0\in \mathbb{R}^{Nd}$ can be arbitrary. Thus, \eqref{argmin approximation primal}--\eqref{argmin approximation dual} become 
\begin{align} 
\label{xk+1 original} &\mb{x}^{k+1} = \mb{x}^k - G^k \mb{z}^k,\\ 
\label{qk+1 original} &\mb{q}^{k+1}=\mb{q}^{k}+\rho \tilde{H} \mb{x}^{k+1}.
\end{align}

With the equivalent form \eqref{q0}--\eqref{qk+1 original} of \eqref{argmin approximation primal}--\eqref{argmin approximation dual}, we now discuss the design of $G^k$ (which indeed determines $B^k$ in the proximal term of \eqref{argmin approximation primal}), aiming at utilizing $G^k$ to accelerate the information fusion throughout the network. We choose $G^k$ to take the following form:
\begin{equation}
G^k = \zeta^k \mb{I}_{Nd} - \eta^k P_{\tau^k}(H), \quad\forall k\ge0,\label{Gk}
\end{equation}
where $\zeta^k>0$ and $P_{\tau^k}(H) \in \mathbb{R}^{Nd\times Nd}$ is a polynomial of $H$ with degree $\tau^k\ge 1$ and with no constant monomial, given by $P_{\tau^k}(H) = \sum_{t=1}^{\tau^k}a_t^kH^t$ for some $a_1^k,\dots, a_{\tau^k}^k\in \mathbb{R}$. Also, suppose the assumptions below hold.

\begin{assumption}\label{assumtion H}
	The matrices $H, \tilde{H} \in \mathbb{S}^{Nd}$ satisfy:
	\begin{itemize}
		\item[(\romannumeral1)] $H, \tilde{H} \succeq 0$ with $\operatorname{Null}(H)=\operatorname{Null}(\tilde{H})=\mathcal{S}$.
		\item[(\romannumeral2)] $H$ and $\tilde{H}$ can commute in matrix multiplication, i.e., $H\tilde{H} = \tilde{H}H$.
	\end{itemize}
\end{assumption}

\begin{assumption}\label{assumption polynomial}
For each $k\geq 0$, the polynomial $P_{\tau^k}(H)=\sum_{t=1}^{\tau^k}a_t^kH^t$ and the parameters $\zeta^k$, $\eta^k$ satisfy:
\begin{itemize}
	\item[(\romannumeral1)] $P_{\tau^k}(H)$ is positive semi-definite for any degree $\tau^k \geq 1$.
	\item[(\romannumeral2)] $\zeta^k >0$, $0<\eta^k<\zeta^k/\bar{\lambda}_{P_{\tau^k}(H)}$.
\end{itemize}
\end{assumption}

Assumption~\ref{assumtion H} is also assumed in \cite{Alghunaim2021AUA}, which enables our algorithm framework to unify several decentralized algorithms. In addition, since $\bar{\lambda}_{P_{\tau^k}(H)}=\sum_{t=1}^{\tau^k}a_t^k \bar{\lambda}_H^t$, Assumption~\ref{assumption polynomial} guarantees that $G^k\succ \mb{O}_{Nd}$ and that $G^k$, $H$ and $\tilde{H}$ are commutative in matrix multiplication. The choices of $\tau^k$ and $a_t^k$ will be discussed in Section~\ref{convergence analysis} and \ref{section discuss}.

Finally, to enable distributed implementation, we choose 
\begin{equation*}
H= P \otimes \mb{I}_d, \quad \tilde{H} = \tilde{P}\otimes \mb{I}_d,
\end{equation*}
where $P, \tilde{P} \in \Re^{N\times N}$ are symmetric positive semi-definite matrices with neighbor-sparse structures, i.e., their $(i,j)$-entries $p_{ij}$ and $\tilde{p}_{ij}$ are zero if $i\neq j$ and $\{i,j\}\notin\mathcal{E}$. The nodes can jointly determine such neighbor-sparse $P$, $\tilde{P}$ satisfying Assumption~\ref{assumtion H} without any centralized coordination \cite{Wu2020AUA}. Thus, the dual update \eqref{qk+1 original} can be executed in a fully distributed way. Moreover, given any $\mb{y}\in \mathbb{R}^{Nd}$, $P_{\tau^k}(H)\mb{y}$ can be evaluated by the nodes via local interactions only, as is described in Oracle~\ref{oracle acc}, which enables distributed computation of $\mb{x}^{k+1}$ according to \eqref{xk+1 original}.

\floatname{algorithm}{Oracle}
\begin{algorithm}[t] 
        \renewcommand{\thealgorithm}{$\mathcal{A}$}
        \caption{Mixing Acceleration} \label{oracle acc}
        \begin{algorithmic}[1]
                \STATE \textbf{Input:} $\mb{y}=(y_1^{\mathsf{T}},\ldots,y_N^{\mathsf{T}})^{\mathsf{T}}\in \Re^{Nd}$, $P\succeq\mathbf{O}_N$, $\tau^k\geq 1$, $\mb{a}^k=(a_1^k,\dots,a_{\tau^k}^k)^{\mathsf{T}}\in\Re^{\tau^k}$.\\ 
                \STATE \textbf{Procedure} $\text{MACC}(\mathbf{y},P,\tau^k,\mb{a}^k)$ \\
				\STATE Each node $i\! \in\! \mathcal{V}$ maintains a variable $y_i^t$ and sets $y_i^0=y_i$.
                \FOR{$t=1:\tau^k$} 
                \STATE Each node $i \in \mathcal{V}$ sends $y_i^{t-1}$ to every neighbor $j\in\mathcal{N}_i$ and then computes $y_i^{t}= \sum_{j\in {\mathcal{N}}_i \cup \{i\}} p_{ij} y_j^{t-1}$.\\
                \ENDFOR
                \STATE \textbf{Output:} Each node $i \in \mathcal{V}$ returns $\sum_{t=1}^{\tau^k}a_t^ky_i^t$, so that $P_{\tau^k}(H)\mb{y}=((\sum_{t=1}^{\tau^k}a_t^ky_1^t)^{\mathsf{T}},\ldots,(\sum_{t=1}^{\tau^k}a_t^ky_N^t)^{\mathsf{T}})^{\mathsf{T}}$.\\
                \STATE \textbf{End procedure}\\
        \end{algorithmic}
\end{algorithm}

Accordingly, \eqref{q0}--\eqref{qk+1 original} with the above setting of $G^k$, $H$, and $\tilde{H}$ can be implemented over $\mathcal{G}$ in a fully distributed way. We call it \emph{\underline{M}ixing-\underline{A}ccelerated \underline{P}rimal-Dual \underline{Pro}ximal} (MAP-Pro) algorithm, because when we set $\tau^k>1$, $G^k$ leads to multiple communication rounds during each iteration $k$, which potentailly accelerates the information mixing process. The detailed distributed implementation of MAP-Pro is presented in Algorithm~\ref{inner primal}, in which Line~8 is achieved via Oracle~\ref{oracle acc} and is fully decentralized.

\floatname{algorithm}{Algorithm}
\begin{algorithm}[b]
        \renewcommand{\thealgorithm}{1}
	\caption{MAP-Pro}    
	\label{inner primal}               
	\begin{algorithmic}[1]
	  \STATE \textbf{Variables:} $\mb{x}^k=((x_1^k)^{\mathsf{T}},\ldots,(x_N^k)^{\mathsf{T}})^{\mathsf{T}}$, $\mb{z}^k=((z_1^k)^{\mathsf{T}},\ldots,(z_N^k)^{\mathsf{T}})^{\mathsf{T}}$, $\mb{q}^k=((q_1^k)^{\mathsf{T}},\ldots,(q_N^k)^{\mathsf{T}})^{\mathsf{T}}$.
	  \STATE \textbf{Parameters:} $\rho>0$, $\theta>0$, $P,\tilde{P}\succeq\mathbf{O}_N$, $\zeta^k>0$, $\eta^k>0$, $\tau^k \geq 1$, $\mb{a}^k\in \mathbb{R}^{\tau^k}$.
		\STATE \textbf{Initialization:}  
		\STATE Each node $i \in \mathcal{V}$ selects $q_i^0 \in \Re ^d$ such that $\sum _{i\in \mathcal{V}} q_i^0 = \mb{0}$ (or simply sets $q_i^0=0$).  
		\STATE Each node $i \in \mathcal{V}$ arbitrarily sets $x_i^0 \in \Re ^d$ and sends $x_i^0$ to every neighbor $j \in \mathcal{N}_i$.
		\FOR{$k \geq 0$} 
			\STATE Each node $i \in \mathcal{V}$ computes $z_i^k =  \rho \sum_{j \in \mathcal{N}_i \cup \{i\}} p_{ij} x_j^k+\nabla f_i(x_i^k) + \theta q_i^k $ and sends it to every  neighbor $j \in \mathcal{N}_i$.
			\STATE The nodes jointly compute $\mathbf{x}^{k+1} = \mb{x}^k - \zeta^k\mb{z}^k + \eta^k \text{MACC}(\mathbf{z}^k,P,\tau^k,\mb{a}^k)$. 
			\STATE Each node $i\! \in\! \mathcal{V}$ sends $x_i^{k+1}$ to every neighbor $j \!\in \!\mathcal{N}_i$. \\
			\STATE Each node $i \in \mathcal{V}$ computes $q_i^{k+1} = q_i^k + \rho \sum_{j\in \mathcal{N}_i \cup \{i\}} \tilde{p}_{ij} x_j^{k+1}$.
		\ENDFOR
	\end{algorithmic}
\end{algorithm}


\section{CONVERGENCE ANALYSIS}\label{convergence analysis}
In this section, we provide the convergence analysis of MAP-Pro under various conditions.

In order to present our convergence results, we introduce the following notations: Let $K = (\mb{I}_N - \frac{1}{N} \mb{1}_N \mb{1}_N ^{\mathsf{T}}) \otimes \mb{I}_d$, $L = \frac{1}{N}(\mb{1}_N \mb{1}_N ^{\mathsf{T}} \otimes \mb{I}_d)$, $\bar{x}^k = \frac{1}{N}(\mb{1}_N^{\mathsf{T}} \otimes \mb{I}_d) \mb{x}^k$, $\bar{\mb{x}}^k = L \mb{x}^k$, $\mb{g}^k = \nabla \tilde{f}(\mb{x}^k)$, $\bar{\mb{g}}^k = L \mb{g}^k$, $\mb{g}_0^k = \nabla \tilde{f}(\bar{\mb{x}}^k)$, $\bar{\mb{g}}_0^k = L \mb{g}_0^k  = \mb{1}_N \otimes \nabla f(\bar{x}^k)$, and $\mb{s}^k=\mb{q}^k+\frac{1}{\theta}\mb{g}_0^k$. Hence, we have $K^2=K$, $LH = \mathbf{O}_{Nd}$, $L\mb{q}^k = L \tilde{H}^{\frac{1}{2}} \mb{v}^k = \mathbf{0}_{Nd}$.

	Let $\tilde{H}, H$ be the Laplacian matrix of the connected graph $\mathcal{G}$. We first establish some results for the Laplacian matrix. From Assumption~\ref{assumtion H}, $\operatorname{Null} (\tilde{H}) = \operatorname{Null} (H) = \operatorname{Null} (K) = \mathcal{S}$, $\bar{\lambda}_K = 1$, we have $K\tilde{H} = \tilde{H}K = \tilde{H}$, $KH = HK = H$, and
	\begin{align}
		\mathbf{O}_{Nd} \preceq \underline{\lambda}_HK \preceq H \preceq \bar{\lambda}_HK, \label{H range}\\
		\mathbf{O}_{Nd} \preceq \underline{\lambda}_{\tilde{H}}K \preceq \tilde{H} \preceq \bar{\lambda}_{\tilde{H}}K. \label{tildeH range}
	\end{align}

	There exists an orthogonal matrix $\tilde{R}:=[r \hat{R}] \in \Re ^{N \times N}$ with $r = \frac{1}{\sqrt{N}} \mb{1}_N$ and $\hat{R} \in \Re ^{N \times (N-1)}$ such that $H = (\tilde{R} \Lambda_{H} \tilde{R}^{\mathsf{T}})\otimes \mb{I}_d$ and $\tilde{H} =(\tilde{R} \Lambda_{\tilde{H}} \tilde{R}^{\mathsf{T}})\otimes \mb{I}_d$, where $\Lambda_{H}$ and $\Lambda_{\tilde{H}}$ are diagonal matrices whose diagonal entries are the eigenvalues of $H$ and $\tilde{H}$, respectively. Here, we define $Q^k := (G^kH)^{\dagger}$ such that $\operatorname{Null} (Q^k) = \mathcal{S}$ and $G^k H Q^k = K$. By introducing an auxiliary matrix $\tilde{D}^k = (\tilde{R}\Lambda_{\tilde{D}^k} \tilde{R}^{\mathsf{T}})\otimes \mb{I}_{d} $, where $\Lambda_{\tilde{D}^k}:=\operatorname{diag}([0,\lambda_2(\tilde{D}^k), \dots, \lambda_N(\tilde{D}^k)])$, we rewrite
	\begin{equation}
		\tilde{H}=\tilde{D}^k G^k H. \label{tilde H}
	\end{equation}
	Note that we can determine $\tilde{D}^k\preceq \mb{I}_{Nd}$ such that the mixing matrix $\tilde{H}$ is time-invariant and satisfies Assumption~\ref{assumtion H}.

	We let $\Lambda_{G^k} = \operatorname{diag}([{\zeta^k}, \lambda_2(G^k), \dots, \lambda_N(G^k)])$, $G^k = ([r \hat{R}]\Lambda_{G^k}[r \hat{R}]^{\mathsf{T}}) \otimes \mb{I}_{d}$, and have matrices $H$, $\tilde{H}$, $Q^k$, $G^k$, $\tilde{D}^k$, $K$ commutative with each other in matrix multiplication. Especially, $\tilde{H}Q^k=Q^k \tilde{H}=\tilde{D}^k$. In addition, we have
	\begin{align}
		\mathbf{O}_{Nd} \preceq \underline{\lambda}_{G^k}K \preceq G^k K \preceq \hat{\lambda}_{G^k}K, \label{Gk range}\\
		\mathbf{O}_{Nd} \preceq \underline{\lambda}_{\tilde{D}^k}K \preceq \tilde{D}^k \preceq \bar{\lambda}_{\tilde{D}^k}K, \label{tildeDk range}
	\end{align}
	where $\hat{\lambda}_{G^k}:= \max \{\lambda_2(G^k), \dots,$ $ \lambda_N(G^k)\}<\zeta^k$.

Inspired by the proof in \cite{yi2022sublinear}, we provide the following Lemma, which generates a descent sequence, consisting of the consensus error and optimality error.

\begin{lemma}\label{lemma descent sequence}
	Suppose Assumptions~\ref{assumption smooth}--\ref{assumption polynomial} hold. Let $\{\mb{x}^k\}$ be the sequence generated by \eqref{q0}--\eqref{qk+1 original}. Then, for any $k \geq 0$,
	\begin{align}
		\tilde{V}^{k+1} \!-\! \tilde{V}^k\leq & \!- \|\mb{x}^k\|^2_{\hat{\lambda}_{G^k} ({\epsilon_3^k} - {\epsilon_4^k} \hat{\lambda}_{G^k})K }\!- \|\mb{s}^k \|^2_{ \hat{\lambda}_{G^k} ({\epsilon_5^k} - \epsilon_6 \hat{\lambda}_{G^k}) K} \notag\\
		&- {\zeta^k}(\epsilon_7 - {\epsilon_{10}^k} {\zeta^k}) \|\bar{\mb{g}}^k\|^2 - \frac{{\zeta^k}}{4}\|\bar{\mb{g}}_0^k\|^2, \displaybreak[0]\label{tilde Vk+1-tilde Vk}
	\end{align}
	where
	\begin{align}
		\label{def tildeV} \tilde{V}^k =& V^k - \|\mb{x}^{k}\|^2_{\hat{\lambda}_{G^k} ({\epsilon_1^k} + {\epsilon_2^k} \hat{\lambda}_{G^k})K },\\ 
		\label{def V} V^k =& \frac{1}{2} \|\mb{x}^k\|^2_K + \frac{1}{2}\|\mb{s}^k\|^2_{\frac{\theta}{\rho} G^kQ^k + G^kHQ^k }\notag\\ 
		&+ \langle \mb{x}^k, K\mb{s}^k \rangle + (f(\bar{x}^k) - f^*), \\
		{\kappa_1^k} =& \frac{\hat{\lambda}_{G^k}}{\underline{\lambda}_{G^k}}\geq 1, \quad \kappa_2 = \frac{\bar{\lambda}_{H}}{\underline{\lambda}_H}\geq 1, \notag\\
		{\epsilon_1^k} =& \rho\bar{\lambda}_{H}(1/(2\kappa_1^k\kappa_2)+\bar{\lambda}_{\tilde{D}^k}),\quad {\epsilon_2^k} = \rho(\theta+\rho\bar{\lambda}_H)\bar{\lambda}_{\tilde{D}^k}\bar{\lambda}_H,\notag\\
		{\epsilon_3^k} =& \rho \bar{\lambda}_{H}(\frac{1}{2{\kappa_1^k}\kappa_2}-\bar{\lambda}_{\tilde{D}^k})-(1+\bar{M}^2+\frac{\theta}{8\kappa_1^k})>0, \notag\\
		{\epsilon_4^k} =& \frac{9}{2}\rho^2\bar{\lambda}^2_H+\rho\bar{\lambda}_{\tilde{D}^k}\bar{\lambda}_{H}+\frac{1}{4}\notag\\
		&+\rho^2\bar{\lambda}_{\tilde{D}^k}\bar{\lambda}^2_{H}+ (3+\frac{3}{2}\rho\bar{\lambda}_{H})\bar{M}^2+\epsilon_2^k,\notag\\
		{\epsilon_5^k} =& \frac{\theta-1}{{2\kappa_1^k}},\quad \epsilon_6 = \frac{7\theta^2}{2}+\frac{1}{2}(3\theta^2+\rho^2\bar{\lambda}_{H}^2),\quad \epsilon_7 = \frac{1}{4},\notag\\
		{\epsilon_8^k} =& \frac{\bar{M}}{2}+\frac{\bar{M}^2}{\theta \underline{\lambda}_{\tilde{D}^k}}(\frac{1}{\rho\underline{\lambda}_{H}}+\frac{1}{\theta}),\notag\\
		{\epsilon_9^k} =& \frac{\bar{M}^2}{\theta \underline{\lambda}_{\tilde{D}^k}}(\frac{1}{\rho^2\underline{\lambda}_{H}^2}+\frac{1}{\theta^2})+\frac{\bar{M}^2}{2\theta^2\rho\underline{\lambda}_{H}},\quad {\epsilon_{10}^k} = {\epsilon_8^k}\!+\!{\epsilon_9^k}/\underline{\lambda}_{G^k}.\notag
	\end{align}
\end{lemma}
\begin{proof}
	See Appendix~\ref{proof lemma descent sequence}.
\end{proof}

Based on Lemma~\ref{lemma descent sequence}, we derive the following convergence result for MAP-Pro under the nonconvex, smooth settings.

\begin{theorem}\label{theorem nonconvex}
	Suppose Assumptions~\ref{assumption smooth}--\ref{assumption polynomial} hold. Let $\{\mb{x}^k\}$ be the sequence generated by \eqref{q0}--\eqref{qk+1 original} under the following conditions: for all $k\geq0$, given $\kappa_1^k, \kappa_2\geq 1$, let $\bar{\lambda}_{\tilde{D}^k} \in (0,\frac{1}{2{\kappa_1^k} \kappa_2})$, $0<\underline{\lambda}_{\tilde{D}^k}\leq \bar{\lambda}_{\tilde{D}^k}$ , $\theta >\max_{k\geq 0}\{4\kappa_1^k\bar{M}^2,1\}$, $\rho > \max_{k\geq 0}\{(1+\bar{M}^2+\frac{\theta}{8\kappa_1^k})/(\frac{ \bar{\lambda}_{H}}{2{\kappa_1^k} \kappa_2} -  \bar{\lambda}_{H}\bar{\lambda}_{\tilde{D}^k}),\theta/\underline{\lambda}_H\}$, $\hat{\lambda}_{G^k} \in (0, \min_{k\geq 0}\{\frac{{\epsilon_3^k}}{{\epsilon_4^k}}, \frac{{\epsilon_5^k}}{\epsilon_6}, (\epsilon_7-{\epsilon_9^k}{\kappa_1^k})/{\epsilon_8^k}, \frac{1}{2{\epsilon_2^k}}(\sqrt{{\epsilon_1^k}^2\!+\!2{\epsilon_2^k}}-{\epsilon_1^k}),{\frac{1}{2{\epsilon_2^k}}(-{\epsilon_1^k}+\sqrt{(\epsilon_1^k)^2+2-\frac{1}{{\xi_1}}})}\})$, ${\zeta^k} \in (\hat{\lambda}_{G^k}, \frac{\epsilon_7}{{\epsilon_{10}^k}})$. Then,
	\begin{align}
		\label{theorem 1 sublinear}&\frac{\sum _{k=0}^{T}W^k}{T+1} \leq \frac{{\delta_1} \hat{V}^0}{\delta_2 (T+1)}, \\
		\label{theorem 1 function error}& (f(\bar{x}^{T+1}) - f^*) \leq {\delta_1} \hat{V}^0,\, \forall T\geq 0,
	\end{align}
	where
	\begin{align}
		\label{V hat}\hat{V}^k =& \|\mb{x}^k\|_K^2 + \|\mb{s}^k\|^2_K + (f(\bar{x}^k) - f^*), \\
		\label{Wk}W^k =& \|\mb{x}^{k} - \bar{\mb{x}}^k\|^2 +  \|\mb{s}^k\|^2_K + \|\bar{\mb{g}}^k\|^2 + \|\bar{\mb{g}}^k_0\|^2,\\
		{\delta_1} =& \max_{k \geq 0}\frac{1}{2}+\xi_1, \text{ with } {\xi_1} = \frac{1}{2}(\frac{\theta}{\rho \bar{\lambda}_{H}}+1)>\frac{1}{2}, \notag\\
		\delta_2  =& \min_{k\in(0,T)} \{\hat{\lambda}_{G^k}({\epsilon_3^k} - {\epsilon_4^k} \hat{\lambda}_{G^k}),\hat{\lambda}_{G^k}({\epsilon_5^k} - \epsilon_6 \hat{\lambda}_{G^k}) , \notag\\
		&{\zeta^k}(\epsilon_7 - {\epsilon_{10}^k} {\zeta^k}),{\zeta^k}/4 \}.\notag
	\end{align}

\end{theorem}
\begin{proof}
	See Appendix~\ref{proof theorem nonconvex}.
\end{proof}

From \eqref{theorem 1 sublinear}, we have $\min_{k\in[0,T]}W^k=\mathcal{O}(1/T)$, which indicates that MAP-Pro converges to a stationary solution at a sublinear rate $\mathcal{O}(1/T)$. Similar sublinear rates have been established in \cite{hong2017prox,sun2018distributed,sun2019distributed,yi2022sublinear,yi2021linear} for smooth nonconvex problems.

Note that for the parameter selections, we require the knowledge of the global information $\bar{M}$, $\bar{\lambda}_H$ and $\underline{\lambda}_H$, which can be collectively found by via local communication between neighboring nodes \cite{Wu2020ASP}. Subsequently, we illustrate how to determine the matrices $G^k$ and $\tilde{H}$ in the iterations \eqref{q0}--\eqref{qk+1 original} under the parameter settings in Theorem~\ref{theorem nonconvex}. Based on \eqref{Gk}, we first select appropriate $a_1^k,\dots, a_{\tau^k}^k\in \mathbb{R}$ for $k\geq 0$ such that $P_{\tau^k}(H)\succeq \mb{O}$, which can be simply satisfied by choosing $a_1^k,\dots, a_{\tau^k}^k>0$, and thus each node is able to locally attain $\bar{\lambda}_{P_{\tau^k}(H)}$ and $\underline{\lambda}_{P_{\tau^k}(H)}$. By fixing $\kappa_1^k \geq 1$, we can successively determine $\bar{\lambda}_{\tilde{D}^k}$, $\theta$, $\rho$, $\hat{\lambda}_{G^k}$ and $\zeta^k$. For such $\kappa_1^k$, there always exists $\eta^k\!\!=\!\!\zeta^k(\kappa_1^k\!\!-\!1)/(\kappa_1^k\bar{\lambda}_{P_{\tau^k}(H)}\!-\!\underline{\lambda}_{P_{\tau^k}(H)})\!<\!\zeta^k/\bar{\lambda}_{P_{\tau^k}(H)}$, which satisfies Assumption~\ref{assumption polynomial}. For the mixing matrix $\tilde{H}$, we can simply set $\tilde{H}=\bar{\alpha}H$, where $\bar{\alpha} \in (0,\underline{\lambda}_{G^k}/(2\kappa_1^k \kappa_2))$. To see this, we let $\Lambda_{\tilde{D}^k}=\operatorname{diag}([0,\bar{\alpha}/\lambda_2(G^k),\dots,\bar{\alpha}/\lambda_N(G^k)])$ such that $\bar{\lambda}_{\tilde{D}^k}<1/(2\kappa_1^k \kappa_2)$. Hence, Assumption~\ref{assumtion H} is satisfied.

Next, we consider the convergence result of MAP-Pro under the P-{\L} condition that satisfies the following assumption.
\begin{assumption}\label{PL condition}
	The global cost function $f(x)$ satisfies the P-{\L} condition with constant $\nu>0$, i.e.,
	\begin{equation}
		\|\nabla f(x)\|^2 \geq 2\nu (f(x) - f^*), \quad \forall x \in \Re ^d. \label{pl}
	\end{equation}
\end{assumption}

Note that the P-{\L} condition is weaker than strong convexity, and can guarantee the global optimum without the convexity. The following theorem establishes the convergence result for MAP-Pro under the additional P-{\L} condition.
\begin{theorem}\label{theorem pl}
	Suppose Assumptions~\ref{assumption smooth}--\ref{PL condition} hold. Let $\{\mb{x}^k\}$ be the sequence generated by \eqref{q0}--\eqref{qk+1 original} with the same parameters given in Theorem~\ref{theorem nonconvex}. Then, 
	\begin{equation}
		\|\mb{x}^k - \bar{\mb{x}}^k\|^2 + f(\bar{x}^k) - f^* \leq (1-\delta)^{k} \frac{\delta_1}{{\delta_3}}\hat{V}^0, \quad \forall k \geq 0,
		\label{theorem 2 linear}
	\end{equation}
	where 
	\begin{align*}
		{\delta_3} =& \min_{k\geq 0}{\xi_2^k}-{\xi_3^k}, \text{ with } 		{\xi_2^k} = \frac{1}{2}\!-\!{\epsilon_1^k}\hat{\lambda}_{G^k}\!-\!{\epsilon_2^k}\hat{\lambda}^2_{G^k}>0,\notag\\
		{\xi_3^k} =& \frac{1}{2}({\xi_2^k}-{\xi_1}+(({\xi_2^k}-{\xi_1})^2+1)^{\frac{1}{2}})>0, \notag\\
		0<&\delta = {\delta_4}/{\delta_1}<1, \text{ with } \notag\\
		\delta_4 =& \min_{k\geq 0} \{\hat{\lambda}_{G^k}({\epsilon_3^k} - \hat{\lambda}_{G^k} {\epsilon_4^k}) , \hat{\lambda}_{G^k}({\epsilon_5^k} - \hat{\lambda}_{G^k}\epsilon_6), \frac{\nu}{2}{\zeta^k}\}.
	\end{align*}
	The definitions of $\delta_1$ is given in Theorem~\ref{theorem nonconvex} and $\epsilon_1^k$, $\epsilon_2^k$, $\epsilon_3^k$, $\epsilon_4^k$, $\epsilon_5^k$, $\epsilon_6$ are defined in Lemma~\ref{lemma descent sequence}.

\end{theorem}
\begin{proof}
	See Appendix~\ref{proof theorem pl}.
\end{proof}

As is shown in \eqref{theorem 2 linear}, MAP-Pro enjoys a linear convergence to the global optimum under the P-{\L} condition. Similar linear rates have been achieved in \cite{Alghunaim2021AUA}, \cite{yi2021linear} and \cite{yi2022sublinear} under the same conditions. 
When the polynomial $P_{\tau^k}(H)$ is omitted, i.e., $G^k=\zeta^k \mb{I}_{Nd}$, MAP-Pro reduces to L-ADMM \cite{yi2022sublinear} with $\zeta^k=\frac{1}{\gamma}, \rho=\alpha, \theta=\beta, H=L, \tilde{H} = \frac{\beta}{\alpha\gamma}L$.

\section{Design of Mixing Acceleration}\label{section discuss}
In this section, we discuss the design of the polynomial $P_{\tau^k}(H)$ to further accelerate the convergence.

 
Here, we let $\kappa_1:=\max_{k\geq 0}\kappa_1^k$. In Theorem~\ref{theorem nonconvex}, we let $\theta=\mathcal{O}(\kappa_1\bar{M}^2)$ and have $\rho=\mathcal{O}(\kappa_1\bar{M}^2/\underline{\lambda}_H)$. Thus, we obtain $\delta_1=\mathcal{O}(\kappa_1\kappa_2)$ and $1/\delta_2=\mathcal{O}(\kappa_1^2\kappa_2^2\bar{M}^2)$. Specifically, $\min_{k\in[0,T]}W^k=\mathcal{O}(\kappa_1^3\kappa_2^3/T)$. Besides, since ${\kappa_1^k}=(\zeta^k-\eta^k\underline{\lambda}_{P_{\tau^k}(H)})/(\zeta^k-\eta^k\bar{\lambda}_{P_{\tau^k}(H)})$, for fixed $\zeta^k$ and $\eta^k$, we can optimize $\kappa_1^k$ by minimizing the eigengap $\kappa_p^k:=\bar{\lambda}_{P_{\tau^k}(H)}/\underline{\lambda}_{P_{\tau^k}(H)}$. On the other hand, the polynomial $P_{\tau^k}(H)$ corresponds to a weighted interaction graph, and smaller ${\kappa_p^k}$ yields higher density of the graph, which facilitates information fusion. Similar statements have been established in \cite{scaman2017optimal} and \cite{xu2020accelerated} (for strongly convex, smooth problems and convex, smooth problems, respectively) that the convergence rates of their algorithms depend on $\kappa_p^k$. Specifically, smaller ${\kappa_p^k}$ induces faster convergence. 

Motivated by this, we introduce the well-known Chebyshev acceleration \cite{auzinger2011iterative}  to our distributed setting, which is capable to generate a favorable polynomial $P_{\tau^k}(H)$ for a fixed $\tau^k$. Subsequently, we demonstrate that the acceleration scheme can be implemented over the graph $\mathcal{G}$, as is detailed in Oracle~\ref{oracle cheb}.

\floatname{algorithm}{Oracle}
\begin{algorithm}[t] 
        \renewcommand{\thealgorithm}{$\mathcal{B}$}
        \caption{Chebyshev Acceleration} \label{oracle cheb}
        \begin{algorithmic}[1]
                \STATE \textbf{Input:} $\mb{y}=(y_1^{\mathsf{T}},\ldots,y_N^{\mathsf{T}})^{\mathsf{T}}\in \Re^{Nd}$, $P\succeq\mathbf{O}_N$, $\tau^k \geq 1$, $c_1 = \frac{\kappa_2+1}{\kappa_2-1}$.\\
                \STATE \textbf{Procedure} $\text{CACC}(\mathbf{y},P,\tau^k)$ \\
                \STATE Each node $i \in \mathcal{V}$ computes $b^0 = 1$, $b^1 = c_1$.\\ 
	        	\STATE Each node $i \in \mathcal{V}$ maintains a variable $y_i^t$, sets $y_i^0 = y_i$ and sends it to every  neighbor $j \in \mathcal{N}_i$.\\
				\STATE Each node $i \in \mathcal{V}$ computes $y_i^1 = c_1y_i^0 - c_1\sum_{j\in {\mathcal{N}}_i \cup \{i\}}$ $ p_{ij} y_j^0$ and sends it to every  neighbor $j \in \mathcal{N}_i$.\\
                \FOR{$t=1:\tau^k -1$} 
                \STATE Each node $i \in \mathcal{V}$ computes $b^{t+1}= 2c_1 b^t-b^{t-1}$.\\
                \STATE Each node $i \in \mathcal{V}$ computes $y_i^{t+1}= 2c_1 y_i^t-y_i^{t-1}-2c_1\sum_{j\in {\mathcal{N}}_i \cup \{i\}} p_{ij} y_j^t$ and sends it to every  neighbor $j \in \mathcal{N}_i$.\\
                \ENDFOR
                \STATE \textbf{Output:} Each node $i \in \mathcal{V}$ returns $y_i^0 - y_i^{\tau^k}/b^{\tau^k}$, so that $P_{\tau^k}(H)\mb{y}=((y_1^0 - y_1^{\tau^k}/b^{\tau^k})^{\mathsf{T}},\ldots,(y_N^0 - y_N^{\tau^k}/b^{\tau^k})^{\mathsf{T}})^{\mathsf{T}}$.\\
                \STATE \textbf{End procedure}\\
        \end{algorithmic}
\end{algorithm}

By substituting $\text{CACC}(\mathbf{z}^k\!,P,\tau^k)$ for $\text{MACC}(\mathbf{z}^k\!,P,\tau^k,\mb{a}^k)$  in Line~8 of Algorithm~\ref{inner primal}, we derive the MAP-Pro equipped with Chebyshev acceleration, referred to as MAP-Pro-CA. 
As is shown in \cite{xu2020accelerated}, the optimal number of iterations in each inner loop is proportional to $\sqrt{\kappa_2}$ that makes $\kappa_p^k=\mathcal{O}(1)$, and thus $\kappa_1=\mathcal{O}(1)$. 

In Table~\ref{table convergence rate}, MAP-Pro-CA outperforms MAP-Pro,  \cite{yi2021linear} and \cite{yi2022sublinear} in the dependence of network topology, i.e., MAP-Pro-CA requires the smallest degree of $\kappa_1$ and $\kappa_2$, so that the network sparsity has the least influence on MAP-Pro-CA.

\begin{table}[t]
	\centering
	\caption{Comparison in convergence rates. Here, we measure $\min_{k\in[0,T]}W^k$ for nonconvex settings and $\|\mb{x}^k - \bar{\mb{x}}^k\|^2 + (f(\bar{x}^k) - f^*)=\mathcal{O}(1-\delta)^k$ for P-{\L} condition. }
	\label{table convergence rate} 
	\begin{tabular}{ccc}
		\toprule  
		&$\min_{k\in[0,T]}W^k$&$\delta$\\ 
		\midrule  
		MAP-Pro& $\mathcal{O}(\frac{\kappa_1^3\kappa_2^3\bar{M}^2}{T})$ &$\mathcal{O}(\frac{1}{\nu\kappa_1^3\kappa_2^3\bar{M}^2})$ \\
		\midrule
		\cite{yi2021linear}, \cite{yi2022sublinear} &$\mathcal{O}(\frac{\kappa_2^4\bar{M}^2\max\{\bar{\lambda}_H^3,1\}}{T})$ & $\mathcal{O}(\frac{1}{\nu\kappa_2^4\bar{M}^2\max\{\bar{\lambda}_H^3,1\}})$\\
		\midrule
		MAP-Pro-CA& $\mathcal{O}(\frac{\kappa_2^3\bar{M}^2}{T})$& $\mathcal{O}(\frac{1}{\nu\kappa_2^3\bar{M}^2})$\\
		\bottomrule  
	\end{tabular}
\end{table}

\section{NUMERICAL EXPERIMENT}\label{section simulation}
In this section, we evaluate the convergence performance of MAP-Pro and MAP-Pro-CA via a numerical example. 

\begin{figure*}[ht]
	\centering
	\begin{minipage}{0.48\linewidth}
	\begin{subfigure}[]{
	    \centering
		\includegraphics[width=3in]{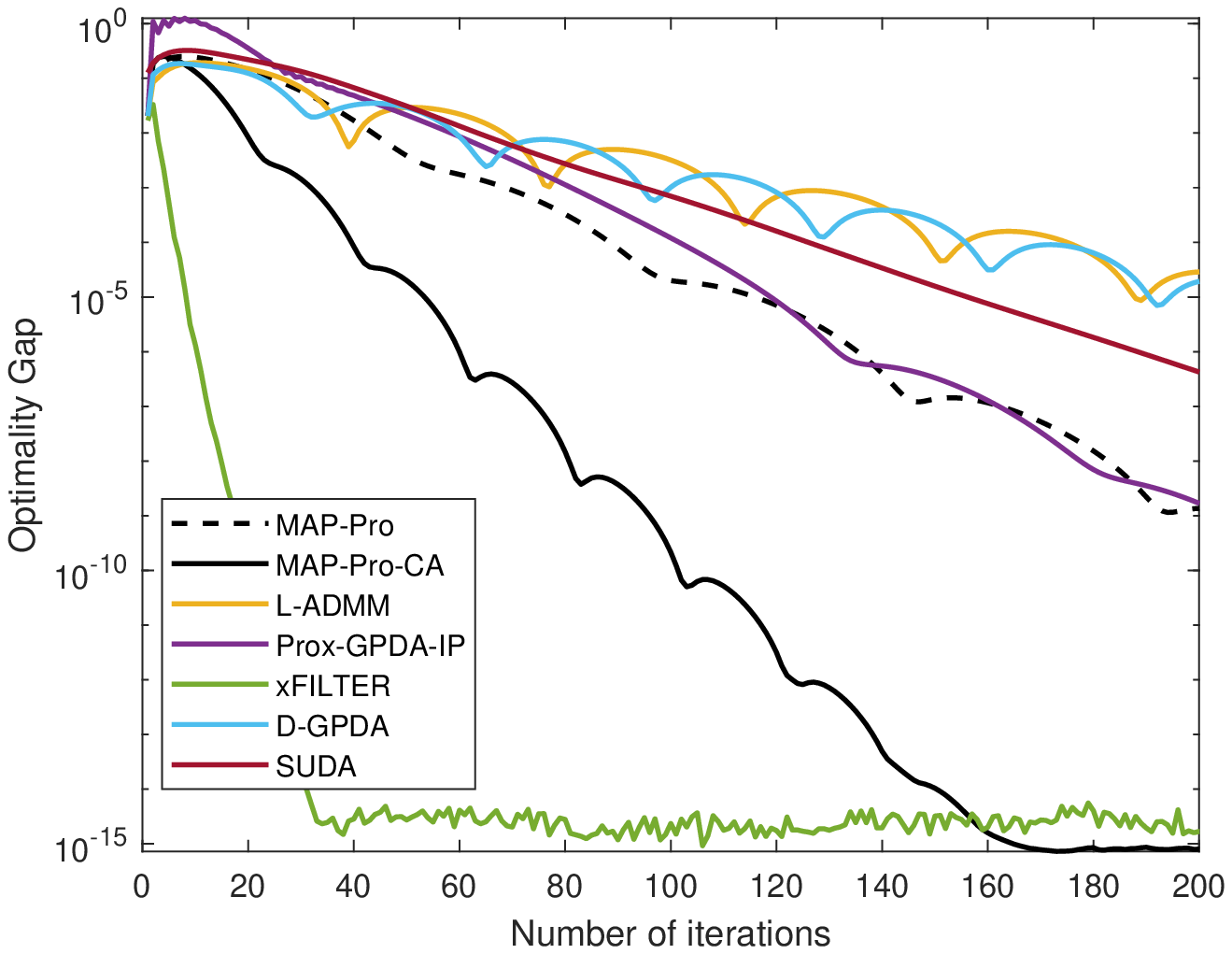}
		\label{iter}}
	\end{subfigure}
	\end{minipage}%
	\begin{minipage}{0.48\linewidth}
	\begin{subfigure}[]{
		\centering
		\includegraphics[width=3in]{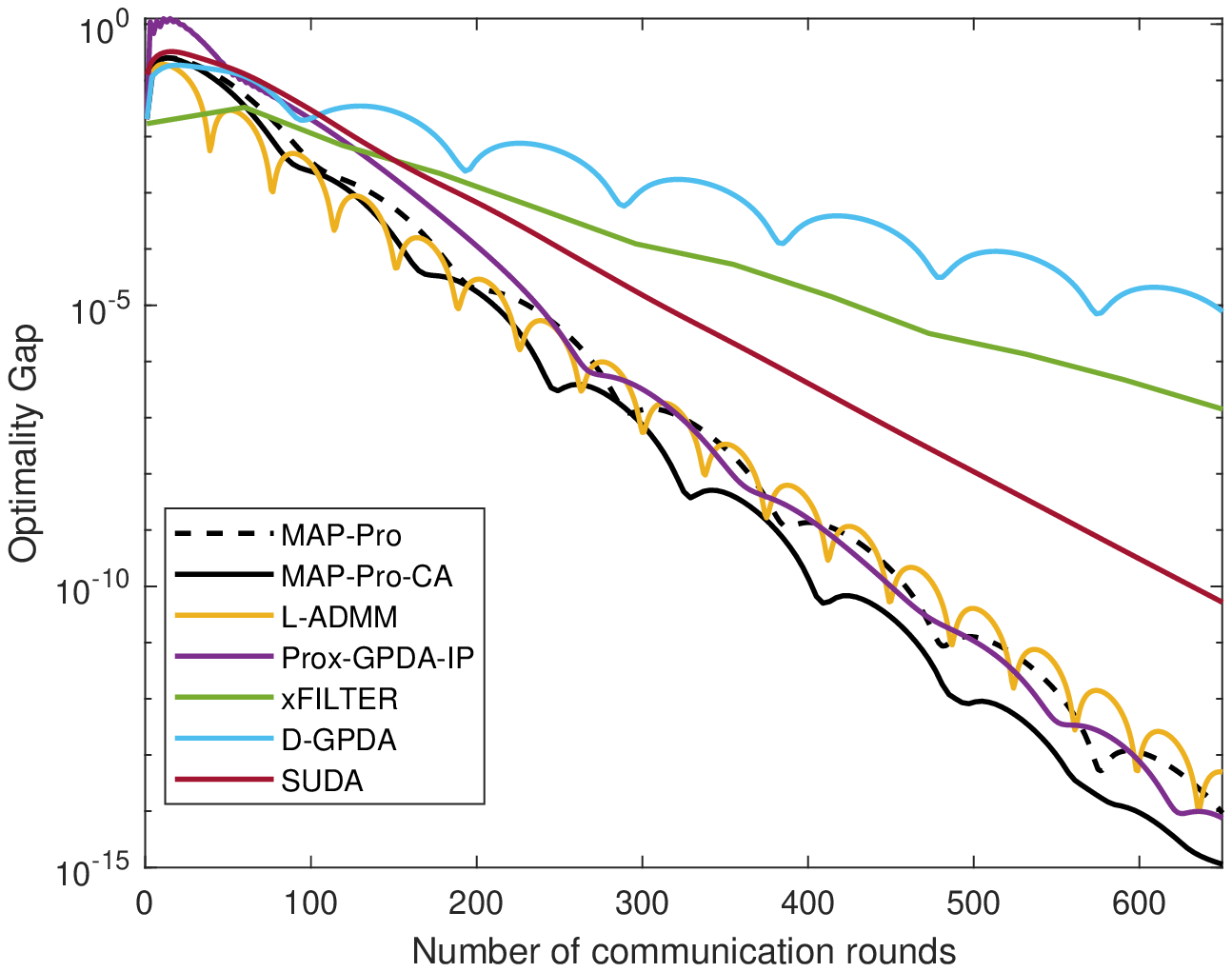}
		\label{comm}
		}\end{subfigure}
	\end{minipage}
	\caption{Convergence performance of MAP-Pro, MAP-Pro-CA, L-ADMM, Prox-GPDA-IP, xFILTER, D-GPDA, and SUDA.\label{figure simulation}}
\end{figure*}

We consider the distributed binary classification problem with nonconvex regularizers \cite{antoniadis2011penalized}, which satisfies Assumptions~\ref{assumption smooth}--\ref{assumption polynomial} and takes the form of \eqref{p1} with
\begin{equation*}
	f_i(x)=\frac{1}{m}\sum_{s}^{m}\log(1+\exp(-y_{is}x^{\mathsf{T}}z_{is}))+\sum_{t=1}^{d}\frac{\lambda\mu([x]_t)^2}{1+\mu([x]_t)^2}.
\end{equation*}
Here, $m$ is the number of data samples of each node, $y_{is}\in \{-1,1\}$ and $z_{is}\in \mathbb{R}^d$ denote the label and the feature for the $s$-th data sample of node $i$, respectively. In the simulation, we set $N=20$, $d=5$, $m=200$ with the regularization parameters $\lambda=0.001$ and $\mu=1$. Besides, we randomly generate $y_{is}$ and $z_{is}$ for each node $i$, and randomly construct the connected graph $\mathcal{G}$ containing 26 edges.

We compare our proposed MAP-Pro and MAP-Pro-CA with the distributed linearized ADMM (L-ADMM) \cite{yi2022sublinear}, the first-order proximal primal-dual algorithms SUDA \cite{Alghunaim2021AUA}, Prox-GPDA-IP \cite{hong2017prox} and D-GPDA \cite{sun2018distributed}, as well as the xFILTER algorithm \cite{sun2019distributed}, which utilizes Chebyshev iterations to approximate the minimization step. In SUDA, we let $A=I-H,B=H^{1/2},C=I$; In MAP-Pro, we let $P_{\tau^k}(H)=H$; in MAP-Pro-CA, we fix $\tau^k=3$ for Chebyshev acceleration; In Prox-GPDA-IP, we set $\beta^k=0.075\log(k)$. We hand-optimize the parameters of these algorithms, and measure their performance by the optimality gap given by 
\vspace{-0.05cm}
\begin{equation*}
	\text{optimality gap:}\quad \|\nabla \tilde{f}(\mb{x})\|^2+\|H^{\frac{1}{2}}\mb{x}\|^2.
\end{equation*}
\vspace{-0.4cm}

Figure~\ref{iter} illustrates the optimality gap versus the number of iterations. As MAP-Pro, MAP-Pro-CA and xFILTER require inner loops in each iteration, which increase the communication cost, we also compare their optimality gaps with respect to the communication rounds in Figure~\ref{comm}. Note that each node transmits a local vector to its neighbors in a communication round. As is shown in Figure~\ref{iter}, xFILTER reaches the optimality within $40$ iterations; however, it requires $58$ inner loops per iteration in this numerical example, so that in Figure~\ref{comm}, its performance is worse than MAP-Pro-CA, MAP-Pro, Prox-GPDA-IP, SUDA and L-ADMM regarding communication cost. In contrary, our proposed MAP-Pro-CA only conducts $3$ inner loops in each primal update. It is faster than MAP-Pro, Prox-GPDA-IP, SUDA, L-ADMM and D-GPDA in terms of the number of iteration and has the best communication efficiency in solving the problem. Besides, MAP-Pro, with slower convergence than MAP-Pro-CA yet a simpler $P_{\tau^k}(H)$, has a similar convergence rate as Prox-GPDA-IP, which requires an increasing penalty parameter, and is faster than SUDA, L-ADMM, x-FILTER and D-GPDA.


\section{CONCLUSIONS}\label{section conclusion}

We have designed MAP-Pro, a mixing-accelerated primal-dual proximal algorithm for addressing nonconvex, smooth distributed optimization. Different from the existing primal-dual methods, MAP-Pro incorporates a time-varying mixing polynomial into an augmented-Lagrangian-like function. Such design intends to facilitate information mixing among the nodes, and is highly efficient especially when the communication cost is mild. We provide sublinear and linear convergence rates of MAP-Pro under various conditions. With Chebyshev acceleration, MAP-Pro-CA improves the convergence results, which outperform those of several existing methods with respect to network topology dependence. Finally, we demonstrate the superior convergence performance of MAP-Pro-CA and MAP-Pro via comparative numerical experiments.

\section{APPENDIX} \label{appendix}
\subsection{Proof of Lemma~\ref{lemma descent sequence}}\label{proof lemma descent sequence}
We first note that $V^k$ is well-defined since $f^* > -\infty$ as assumed in Assumption~\ref{a1}.

From \eqref{Gk} we have 
	\begin{align}
		LG^k = & L(\zeta^k \mathbf{I}_{Nd}-\eta^k P_{\tau^k}(H))=\zeta^k L. \label{TGk}
	\end{align}
It follows from \eqref{zk original}, \eqref{xk+1 original}, \eqref{TGk} and $LH=0$ that 
\begin{align}
	\bar{\mb{x}}^{k+1} - \bar{\mb{x}}^k = -L G^k(\mb{g}^k + \rho H \mb{x}^k + \mb{q}^k) = - {\zeta^k}\bar{\mb{g}}^k. \label{bar xk+1 - bar xk}
\end{align}

By virtue of the smoothness of the global cost function $\tilde{f}$,
\begin{align}
	\|\mb{g}_0^k - \mb{g}^k\|^2 \leq \bar{M}^2\|\bar{\mb{x}}^k - \mb{x}^k\|^2 = \bar{M}^2 \|\mb{x}^k\|^2_K. \label{g0k-gk}
\end{align}
Then, since $\bar{\lambda}_{L} = 1$, we have
\begin{align}
	\|\bar{\mb{g}}_0^k - \bar{\mb{g}}^k\|^2 = & \|L(\mb{g}_0^k - \mb{g}^k)\|^2 \leq \bar{M}^2 \|\mb{x}^k\|^2_K. \label{bar g0k- bar gk}
\end{align}
From \eqref{smooth} and \eqref{bar xk+1 - bar xk}, we have
\begin{align}
	\|\mb{g}_0^{k+1} \!- \mb{g}_0^k\|^2 \!\leq\! \bar{M}^2 \|\bar{\mb{x}}^{k+1} - \bar{\mb{x}}^{k}\|^2 \!=\! (\zeta^k)^2 \bar{M}^2\|\bar{\mb{g}}^k\|^2. \label{g0k+1-g0k}
\end{align}
Then, from \eqref{xk+1 original}, \eqref{H range}, \eqref{Gk range} and \eqref{g0k-gk}, we obtain
\begin{align}
	&\|\mb{x}^{k+1}-\mb{x}^k\|^2_{K}= \|G^k(\theta\mb{q}^k +\mb{g}_0^k + \rho H \mb{x}^k + \mb{g}^k-\mb{g}_0^k)\|^2_K \notag\\
	\leq & 3\|\mb{s}^k\|^2_{\theta^2{G^k}^2K} + 3\|\mb{x}^k\|^2_{\rho^2{G^k}^2H^2} + 3\|\mb{g}^k - \mb{g}_0^k\|^2_{{G^k}^2K}\notag\\
	\leq & 3\|\mb{s}^k\|^2_{\theta^2\hat{\lambda}^2_{G^k}K} + 3\|\mb{x}^k\|^2_{\rho^2\hat{\lambda}^2_{G^k}\bar{\lambda}^2_{H}K+\hat{\lambda}^2_{G^k}\bar{M}^2K}. \label{xk+1 - xk}
\end{align}
Since $\bar{\lambda}_K =1$, it then follows from \eqref{zk original} and \eqref{xk+1 original} that
\begin{align}
	&\frac{1}{2}\|\mb{x}^{k+1}\|_K^2 \notag\\
	= & \frac{1}{2}\|\mb{x}^k - G^k(\rho H \mb{x}^k + \theta\mb{q}^k +\mb{g}_0^k + \mb{g}^k-\mb{g}_0^k)\|_K^2 \notag\\
	\leq & \frac{1}{2} \|\mb{x}^{k}\|_K^2 -\langle  \mb{x}^k -\mb{x}^{k+1}+\mb{x}^{k+1},\theta G^k K\mb{s}^k \rangle \notag\\
	& -\! \|\mb{x}^k\|^2_{\rho G^k H - \frac{1}{2} \rho^2 {G^k}^2 H^2}\!+ \!\frac{1}{2}\|\mb{x}^{k}\|_{G^k K}^2\! + \!\frac{1}{2} \|\mb{g}_0^k - \mb{g}^k\|^2_{G^k K} \notag\\
	& + \frac{\rho^2}{2} \|\mb{x}^k\|^2_{{G^k}^2 H^2} + \frac{\theta^2}{2}\|\mb{s}^k\|^2_{{G^k}^2K} + \frac{\rho^2}{2} \|\mb{x}^k\|^2_{{G^k}^2 H^2}\notag\\
	&  + \frac{1}{2} \|\mb{g}_0^k - \mb{g}^k\|^2_{{G^k}^2 K}+ \theta^2\|\mb{s}^k\|^2_{{G^k}^2K} + \|\mb{g}_0^k - \mb{g}^k\|^2_{{G^k}^2K} \notag\\
	\leq & \frac{1}{2} \|\mb{x}^{k}\|_K^2 - \|\mb{x}^k\|^2_{\rho G^k H - \frac{1}{2} G^k K - \frac{3}{2} \rho^2 {G^k}^2 H^2} \notag\\
	&- \langle  \mb{x}^{k+1},\theta G^k K\mb{s}^k \rangle + \frac{1}{2}\|\mb{x}^{k+1}-\mb{x}^{k}\|_K^2 \notag\\
	&+ 2\theta^2\|\mb{s}^k\|^2_{{G^k}^2K} + \frac{1}{2}\|\mb{g}_0^k - \mb{g}^k\|^2_{G^kK+3{G^k}^2K}\notag\\
	\leq & \frac{1}{2} \|\mb{x}^{k}\|_K^2-\langle  \mb{x}^{k+1},\theta G^k K\mb{s}^k \rangle+ \frac{7\theta^2}{2}\|\mb{s}^k\|^2_{\hat{\lambda}^2_{G^k}K}\notag\\
	&\!-\! \|\mb{x}^k\|^2_{(\rho \underline{\lambda}_H\underline{\lambda}_{G^k} \!- \frac{1}{2}(1+\bar{M}^2)\hat{\lambda}_{G^k}\!-3\rho^2\bar{\lambda}^2_{H}\hat{\lambda}^2_{G^k}\!-3\bar{M}^2\hat{\lambda}^2_{G^k})K},
	\label{norm xk+1}
\end{align}
where the first two inequalities hold since the Cauchy-Schwarz inequality, and the last inequality holds since \eqref{H range}, \eqref{Gk range}, \eqref{g0k-gk}, \eqref{xk+1 - xk} and $\bar{\lambda}_K=1$.

For the second term of \eqref{def V}, we have
\begin{align}
	&\!\!\frac{1}{2} \|\mb{s}^{k+1}\|^2_{\frac{\theta}{\rho}G^{k+1}Q^{k+1}+ G^{k+1}HQ^{k+1}}\!=\!\frac{1}{2} \|\mb{s}^{k+1}\|^2_{\frac{\theta}{\rho}G^{k}Q^{k}+ G^{k}HQ^{k}} \notag\\
	\leq& \frac{1}{2}\|\mb{s}^{k} + \rho \tilde{H} \mb{x}^{k+1} + \frac{1}{\theta}(\mb{g}_0^{k+1}-\mb{g}_0^k) \|^2_{(\frac{\theta}{\rho}G^kQ^k+ G^kHQ^k)(\tilde{D}^k)^{\dagger} } \notag\\
	\leq & \frac{1}{2}\|\mb{s}^{k} \|^2_{\frac{\theta}{\rho}G^kQ^k+ G^kHQ^k} + \langle \mb{s}^{k}, (\theta G^kK+\rho G^kH)\mb{x}^{k+1} \rangle \notag\\
	&  + \frac{1}{2}\|\mb{s}^{k}\|^2_{\theta{G^k}K} + \|\mb{g}_0^{k+1}-\mb{g}_0^k\|^2_{(\frac{K}{\rho^2\theta}+\frac{H^2}{\theta^3})G^k(\tilde{H}^{\dagger})^2}\notag\\
	&+ \|\mb{x}^{k+1}\|^2_{(\theta \mb{I} +\rho H)\rho G^k\tilde{H}}+\|\mb{g}_0^{k+1}-\mb{g}_0^k\|^2_{(\frac{K}{\rho\theta}+\frac{H}{\theta^2})G^k\tilde{H}^{\dagger}}\notag\\
	\leq& \frac{1}{2}\|\mb{s}^{k} \|^2_{\frac{\theta}{\rho}G^kQ^k+ G^kHQ^k} + \langle \mb{s}^{k}, (\theta G^kK+\rho G^kH)\mb{x}^{k+1} \rangle \notag\\
	& + \|\mb{x}^{k+1}\|^2_{(\theta +\rho \bar{\lambda}_H)\rho\bar{\lambda}_{\tilde{D}^k}\bar{\lambda}_{H}\hat{\lambda}_{G^k}^2 K} + \frac{1}{2}\|\mb{s}^{k}\|^2_{\theta{G^k}K}\notag\\
	&+\!\frac{1}{\theta \underline{\lambda}_{\tilde{D}^k}}\Big(\frac{1}{\rho \underline{\lambda}_{H}}\!+\!\frac{1}{\theta}\!+\!\frac{1}{ \underline{\lambda}_{\tilde{D}^k}\underline{\lambda}_{G^k}}(\frac{1}{\rho^2\underline{\lambda}_{H}^2}\!+\!\frac{1}{\theta^2}) \Big){\zeta^k}^2 \bar{M}^2\|\bar{\mb{g}}^k\|^2,
	\label{norm sk+1}
\end{align}
where the first equality holds since $G^kQ^k=H^{\dagger}$; the first inequality follows from \eqref{qk+1 original} and $(\tilde{D}^k)^{\dagger} \succeq \mb{I}_{Nd}$; the second inequality holds since $Q^k(\tilde{D}^k)^{\dagger}=\tilde{H}^{\dagger}$, and the last inequality holds since \eqref{H range}--\eqref{tildeDk range}, \eqref{g0k+1-g0k} and $\bar{\lambda}_K=1$.

For the third term of \eqref{def V}, we have
\begin{align}
	& \langle \mb{x}^{k+1}, K\mb{s}^{k+1} \rangle \notag\\
	=& \langle \mb{x}^{k} - G^k(\rho H \mb{x}^k + \theta\mb{q}^k +\mb{g}_0^k + \mb{g}^k-\mb{g}_0^k), K(\mb{q}^k +\frac{1}{\theta}\mb{g}_0^k) \rangle \notag\\
	& + \langle \mb{x}^{k+1}, K(\rho \tilde{H}\mb{x}^{k+1} +\frac{1}{\theta} (\mb{g}_0^{k+1} - \mb{g}_0^k)) \rangle \notag\\
	\leq& \langle \mb{x}^k, K\mb{s}^k\rangle - \langle \mb{x}^k-\mb{x}^{k+1}+\mb{x}^{k+1}, \rho G^kH\mb{s}^k\rangle\notag\\
	&-\! \|\mb{s}^k\|^2_{\theta G^kK}\!+\! \frac{1}{2}\|\mb{g}^k\!-\!\mb{g}_0^k\|^2_{G^kK}\! +\!\frac{1}{2}\|\mb{s}^k\|^2_{G^kK} \!+ \!\|\mb{x}^{k+1}\|^2_{\rho\tilde{H}} \notag\\
	&+\frac{\rho \underline{\lambda}_{G^k}\underline{\lambda}_{H}}{2}\|\mb{x}^{k+1}\|^2_{K} + \frac{1}{2\theta^2\rho \underline{\lambda}_{G^k}\underline{\lambda}_{H}}\|\mb{g}_0^{k+1} - \mb{g}_0^k\|^2\notag\\
	\leq & \langle \mb{x}^k, K\mb{s}^k\rangle - \langle \mb{x}^{k+1}, \rho G^kH\mb{s}^k\rangle+ \frac{1}{2}\|\mb{x}^{k+1}-\mb{x}^k\|^2_K\notag\\
	&  - \|\mb{s}^k\|^2_{(\theta-\frac{1}{2}) G^kK -\frac{1}{2}\rho^2 {G^k}^2H^2} + \frac{1}{2}\|\mb{g}^k-\mb{g}_0^k\|^2_{G^k K}\notag\\
	& +\|\mb{x}^{k+1}\|^2_{\frac{\rho \underline{\lambda}_{G^k}\underline{\lambda}_{H}}{2}K+\rho \tilde{H}} + \frac{1}{2\theta^2\rho \underline{\lambda}_{G^k}\underline{\lambda}_{H}}\|\mb{g}_0^{k+1} - \mb{g}_0^k\|^2\notag\\
	\leq & \langle \mb{x}^k, K\mb{s}^k\rangle - \langle \mb{x}^{k+1}, \rho G^kH\mb{s}^k\rangle+\frac{(\zeta^k)^2\bar{M}^2}{2\theta^2\rho \underline{\lambda}_{G^k}\underline{\lambda}_{H}}\|\bar{\mb{g}}^k\|^2 \notag\\
	&+ \frac{1}{2}\|\mb{x}^{k}\|^2_{(\bar{M}^2\hat{\lambda}_{G^k}+3(\rho^2\bar{\lambda}^2_H+\rho\bar{\lambda}_{H}\bar{M}^2)\hat{\lambda}^2_{G^k})K}  \notag\\
	&- \|\mb{s}^k\|^2_{(\theta-\frac{1}{2}){G^k}K -\frac{1}{2}(3\theta^2+\rho^2\bar{\lambda}_{H}^2) \hat{\lambda}^2_{G^k}K}\notag\\
	&+ \|\mb{x}^{k+1}\|^2_{\frac{\rho \underline{\lambda}_{G^k}\underline{\lambda}_{H}}{2}K+\rho \bar{\lambda}_{\tilde{D}^k}\bar{\lambda}_{H}\hat{\lambda}_{G^k}K},
	\label{xk+1, sk+1}
\end{align}
where the first equality holds since iterations \eqref{zk original}--\eqref{qk+1 original} and the first two inequalities hold since the Cauchy-Schwarz inequality; the last inequality holds since \eqref{H range}--\eqref{tildeDk range}, \eqref{g0k-gk}, \eqref{g0k+1-g0k}, \eqref{xk+1 - xk}, and $\bar{\lambda}_K=1$.

For the last term of \eqref{def V}, we have
\begin{align}
	&f(\bar{x}^{k+1}) - f^*= \tilde{f}(\bar{\mb{x}}^{k+1}) - \tilde{f}^* \notag\\
	= & \tilde{f}(\bar{\mb{x}}^{k}) - \tilde{f}^* + \tilde{f}(\bar{\mb{x}}^{k+1}) - \tilde{f}(\bar{\mb{x}}^{k}) \notag\\
	\leq & \tilde{f}(\bar{\mb{x}}^{k}) - \tilde{f}^* - \langle {\zeta^k}\bar{\mb{g}}^k, \bar{\mb{g}}_0^k \rangle + \frac{1}{2}(\zeta^k)^2\bar{M}\|\bar{\mb{g}}^k\|^2\notag\\
	= & f(\bar{x}^k) -f^* - \frac{1}{2}{\zeta^k}\langle \bar{\mb{g}}^k, \bar{\mb{g}}^k +\bar{\mb{g}}_0^k -\bar{\mb{g}}^k \rangle \notag\\
	&- \frac{1}{2}{\zeta^k}\langle \bar{\mb{g}}^k -\bar{\mb{g}}_0^k+\bar{\mb{g}}_0^k, \bar{\mb{g}}^k_0  \rangle + \frac{1}{2}(\zeta^k)^2\bar{M}\|\bar{\mb{g}}^k\|^2\notag\\
	\leq & f(\bar{x}^k) -f^* - \frac{{\zeta^k}}{4}\|\bar{\mb{g}}^k\|^2 + \frac{{\zeta^k}}{2} \|\bar{\mb{g}}^k -\bar{\mb{g}}_0^k\|^2\notag\\
	&  + \frac{{\zeta^k}}{4} \|\bar{\mb{g}}_0^k\|^2 - \frac{{\zeta^k}}{2}\|\bar{\mb{g}}_0^k\|^2+ \frac{1}{2}(\zeta^k)^2\bar{M}\|\bar{\mb{g}}^k\|^2 \notag\\
	\leq & f(\bar{x}^k) -f^* - \frac{{\zeta^k}}{4}(1-2{\zeta^k}\bar{M})\|\bar{\mb{g}}^k\|^2 \notag\\
	&- \frac{{\zeta^k}}{4}\|\bar{\mb{g}}_0^k\|^2 + \frac{1}{2}{\zeta^k}\bar{M}^2\|\mb{x}^k\|^2_K,
	\label{fk+1-f*}
\end{align}
where the first equality holds for $f^* = \tilde{f}^*$. The first inequality holds since the descent lemma, \eqref{bar xk+1 - bar xk} and $L=L^2$. The second inequality follows from the Cauchy-Schwarz inequality and the last inequality holds since \eqref{bar g0k- bar gk}.

Here we let ${\zeta^k}<\frac{\theta}{4\kappa_1^k\bar{M}^2}\hat{\lambda}_{G^k}$. Substituting this to the last term in \eqref{fk+1-f*} and combining \eqref{norm xk+1}--\eqref{fk+1-f*} yield \eqref{tilde Vk+1-tilde Vk}.

\subsection{Proof of Theorem~\ref{theorem nonconvex}}\label{proof theorem nonconvex}
From \eqref{def tildeV} and \eqref{def V}, we have
\begin{align}
	\tilde{V}^k \geq & {\xi_2^k} \|\mb{x}^k\|^2_K + {\xi_1} \|\mb{s}^k\|^2_K - {\xi_3^k}\|\mb{x}^k\|^2_K \notag\\
	&- {\|\mb{s}^k\|^2_K}/({4{\xi_3^k}}) + (f(\bar{x}^k) - f^*) \notag\\
	= & ({\xi_2^k}\!-\!{\xi_3^k})(\|\mb{x}^k\|^2_K\!\! + \!\|\mb{s}^k\|^2_K)\!+\!f(\bar{x}^k) \!- \!f^*\!\geq\!  {\delta_3} \hat{V}^k, \label{delta1}
\end{align}
where ${\xi_2^k},\xi_3^k$ and $\delta_3$ are given in Theorem~\ref{theorem pl}. The first inequality follows from \eqref{tilde H} and the AM-GM inequality; the second equality holds since ${\xi_2^k}-{\xi_3^k}={\xi_1}-\frac{1}{4{\xi_3^k}}$. From $\hat{\lambda}_{G^k}<(\sqrt{{\epsilon_1^k}^2+2{\epsilon_2^k}}-{\epsilon_1^k})/(2{\epsilon_2^k})$, we have ${\xi_2^k}>0$. For the last inequality, since ${\xi_1} > \frac{1}{2}$, we have $2-\frac{1}{{\xi_1}}>0$, together with $\hat{\lambda}_{G^k}<\frac{1}{2{\epsilon_2^k}}(-{\epsilon_1^k}+\sqrt{(\epsilon_1^k)^2+2-\frac{1}{{\xi_1}}})$, we obtain ${\xi_1}{\xi_2^k}>\frac{1}{4}$. Thus, the last inequality in \eqref{delta1} follows from $0<{\xi_2^k}-{\xi_3^k}<\frac{1}{2}$.

Similarly, for $\delta_1>1$, we have
\begin{equation}
	\tilde{V}^k \leq {\delta_1} \hat{V}^k. \label{delta2}
\end{equation}

From $0<\bar{\lambda}_{\tilde{D}^k} < \frac{1}{2{\kappa_1^k} \kappa_2}\leq \frac{1}{2}$, ${\kappa_1^k},\kappa_2> 0$, we have $0 < \frac{1}{2{\kappa_1^k} \kappa_2} - \bar{\lambda}_{\tilde{D}^k} < 1$. From $\rho > (1+\bar{M}^2+\frac{\theta}{8\kappa_1})/(\frac{ \bar{\lambda}_{H}}{2{\kappa_1^k} \kappa_2} -  \bar{\lambda}_{H}\bar{\lambda}_{\tilde{D}^k})$, we have ${\epsilon_3^k} > 0$.

Next, from $0 < \hat{\lambda}_{G^k} < \min\{\frac{{\epsilon_3^k}}{{\epsilon_4^k}}, \frac{{\epsilon_5^k}}{\epsilon_6}\}$, we have
\begin{align}
	\label{epsilon3-epsilon4} & \hat{\lambda}_{G^k} ({\epsilon_3^k} - {\epsilon_4^k} \hat{\lambda}_{G^k})> 0 , \\
	\label{epsilon5-epsilon6} & \hat{\lambda}_{G^k}({\epsilon_5^k} - \epsilon_6 \hat{\lambda}_{G^k})> 0.
\end{align}

From $\theta >\max_{k\geq 0}\{4\kappa_1^k\bar{M}^2,1\}$ and $\rho > \theta/\underline{\lambda}_H$, we have ${\epsilon_9^k}<\epsilon_7/\kappa_1^k$. Then, from $\hat{\lambda}_{G^k}<(\epsilon_7-{\epsilon_9^k}{\kappa_1^k})/{\epsilon_8^k}$, there exists ${\zeta^k}$ satisfying $\hat{\lambda}_{G^k}<{\zeta^k}<\frac{\epsilon_7}{\epsilon_{10}^k}$  such that
\begin{align}
	{\zeta^k}(\epsilon_7 - {\epsilon_{10}^k} {\zeta^k}) >0 .\label{epsilon7-epsilon10}
\end{align}
Summing \eqref{tilde Vk+1-tilde Vk} from $k=0$ to $T$ with \eqref{Wk} yields
\begin{equation}
	\sum _{k=0}^T \tilde{V}^{k+1} \leq \sum_{k=0}^T \tilde{V}^k - \delta_2 \sum_{k=0}^T W^k.
	\label{delta3}
\end{equation}

Then, \eqref{delta3} yields
\begin{equation}
	\tilde{V}^{T+1} + \delta_2 \sum_{k=0}^T W^k \leq\tilde{V}^0. \label{tildeVT+1+delta3}
\end{equation}

From \eqref{V hat}, \eqref{delta2} \eqref{delta3} and \eqref{tildeVT+1+delta3}, we obtain \eqref{theorem 1 sublinear}.

Similarly, combining \eqref{def V}, \eqref{delta2} and \eqref{delta3} yields \eqref{theorem 1 function error}.

\subsection{Proof of Theorem~\ref{theorem pl}}\label{proof theorem pl}
From \eqref{delta1} we have
\begin{equation}
	\|\mb{x}^k - \bar{\mb{x}}^k\|^2 + f(\bar{x}^k) - f^* \leq \hat{V}^k \leq \frac{V^k}{{\delta_3}}.\label{norm consensus+function}
\end{equation}
Assumption~\ref{PL condition} implies
\begin{equation}
	\|\bar{\mb{g}}_0^k\|^2 = \|\nabla f(\bar{x}^k)\|^2 \geq 2\nu (f(\bar{x}^k) - f^*).\label{barg0k geq 2nu}
\end{equation}
Based on \eqref{epsilon3-epsilon4} and \eqref{epsilon5-epsilon6}, we have
\begin{equation}
	\delta_4 > 0 \,\,\text{and}\,\, \delta = \frac{\delta_4}{\delta_1} >0. \label{delta4}
\end{equation}
Then from \eqref{tilde Vk+1-tilde Vk}, \eqref{delta2}, \eqref{epsilon7-epsilon10}, \eqref{barg0k geq 2nu} and \eqref{delta4}, we have
\begin{align}
	\tilde{V}^{k+1} \!\leq & \tilde{V}^k- \|\mb{x}^k\|^2_{\hat{\lambda}_{G^k} ({\epsilon_3^k} - {\epsilon_4^k} \hat{\lambda}_{G^k})K } - \|\mb{s}^k\|^2_{  \hat{\lambda}_{G^k}({\epsilon_5^k} - \epsilon_6 \hat{\lambda}_{G^k})K }\notag\\
	& \!\!\!-\! \frac{\nu {\zeta^k}}{2} (f(\bar{x}^k)\! - \!f^*) \!\leq\! \tilde{V}^{k} \!-\! \delta_4 \hat{V}^k\!\leq\! \tilde{V}^k \!\!- \!\frac{\delta_4}{{\delta_1}}\tilde{V}^k,
	\label{tildeVk+1 leq -delta4}
\end{align}
together with \eqref{delta4}, we have
\begin{equation}
	\tilde{V}^{k+1} \!\leq\! (1-\delta) \tilde{V}^k \!\leq\! (1-\delta)^{k+1} \tilde{V}^0 \!\leq\! (1-\delta)^{k+1} {\delta_1} \hat{V}^0.
	\label{tildeVk+1 leq hat V0}
\end{equation}

Combining \eqref{norm consensus+function} and \eqref{tildeVk+1 leq hat V0} yields \eqref{theorem 2 linear}.

\bibliographystyle{IEEEtran}
\bibliography{IEEEabrv,reference}

\begin{thebibliography}{10}
\providecommand{\url}[1]{#1}
\csname url@samestyle\endcsname
\providecommand{\newblock}{\relax}
\providecommand{\bibinfo}[2]{#2}
\providecommand{\BIBentrySTDinterwordspacing}{\spaceskip=0pt\relax}
\providecommand{\BIBentryALTinterwordstretchfactor}{4}
\providecommand{\BIBentryALTinterwordspacing}{\spaceskip=\fontdimen2\font plus
\BIBentryALTinterwordstretchfactor\fontdimen3\font minus
  \fontdimen4\font\relax}
\providecommand{\BIBforeignlanguage}[2]{{%
\expandafter\ifx\csname l@#1\endcsname\relax
\typeout{** WARNING: IEEEtran.bst: No hyphenation pattern has been}%
\typeout{** loaded for the language `#1'. Using the pattern for}%
\typeout{** the default language instead.}%
\else
\language=\csname l@#1\endcsname
\fi
#2}}
\providecommand{\BIBdecl}{\relax}
\BIBdecl

\bibitem{Nedi2015ConvergenceRO}
A.~Nedich, ``Convergence rate of distributed averaging dynamics and
  optimization in networks,'' \emph{Foundations and Trends in Systems and
  Control}, vol.~2, no.~1, pp. 1--100, 2015.

\bibitem{Nedi2009DistributedSM}
A.~Nedi{\'c} and A.~E. Ozdaglar, ``Distributed subgradient methods for
  multi-agent optimization,'' \emph{IEEE Transactions on Automatic Control},
  vol.~54, pp. 48--61, 2009.

\bibitem{Yuan2013OnTC}
K.~Yuan, Q.~Ling, and W.~Yin, ``On the convergence of decentralized gradient
  descent,'' \emph{SIAM Journal on Optimization}, vol.~26, no.~3, pp.
  1835--1854, 2016.

\bibitem{Shi2014EXTRAAE}
W.~Shi, Q.~Ling, G.~Wu, and W.~Yin, ``Extra: An exact first-order algorithm for
  decentralized consensus optimization,'' \emph{SIAM Journal on Optimization},
  vol.~25, no.~2, pp. 944--966, 2015.

\bibitem{Wu2020AUA}
X.~Wu and J.~Lu, ``A unifying approximate method of multipliers for distributed
  composite optimization,'' \emph{IEEE Transactions on Automatic Control},
  vol.~68, no.~4, pp. 2154--2169, 2023.

\bibitem{Wai2015ACD}
H.-T. Wai, T.-H. Chang, and A.~Scaglione, ``A consensus-based decentralized
  algorithm for non-convex optimization with application to dictionary
  learning,'' in \emph{2015 IEEE International Conference on Acoustics, Speech
  and Signal Processing}, 2015, pp. 3546--3550.

\bibitem{Bianchi2011ConvergenceOA}
P.~Bianchi and J.~Jakubowicz, ``Convergence of a multi-agent projected
  stochastic gradient algorithm for non-convex optimization,'' \emph{IEEE
  Transactions on Automatic Control}, vol.~58, no.~2, pp. 391--405, 2013.

\bibitem{hashempour2021distributed}
S.~Hashempour, A.~A. Suratgar, and A.~Afshar, ``Distributed nonconvex
  optimization for energy efficiency in mobile ad hoc networks,'' \emph{IEEE
  Systems Journal}, vol.~15, no.~4, pp. 5683--5693, 2021.

\bibitem{Alghunaim2021AUA}
S.~A. Alghunaim and K.~Yuan, ``A unified and refined convergence analysis for
  non-convex decentralized learning,'' \emph{IEEE Transactions on Signal
  Processing}, vol.~70, pp. 3264--3279, 2021.

\bibitem{scutari2013decomposition}
G.~Scutari, F.~Facchinei, P.~Song, D.~P. Palomar, and J.-S. Pang,
  ``Decomposition by partial linearization: Parallel optimization of
  multi-agent systems,'' \emph{IEEE Transactions on Signal Processing},
  vol.~62, no.~3, pp. 641--656, 2014.

\bibitem{sun2016distributed}
Y.~Sun, G.~Scutari, and D.~Palomar, ``Distributed nonconvex multiagent
  optimization over time-varying networks,'' in \emph{2016 50th Asilomar
  Conference on Signals, Systems and Computers}, 2016, pp. 788--794.

\bibitem{hong2017prox}
M.~Hong, D.~Hajinezhad, and M.-M. Zhao, ``Prox-{PDA}: The proximal primal-dual
  algorithm for fast distributed nonconvex optimization and learning over
  networks,'' in \emph{Proceedings of the 34th International Conference on
  Machine Learning}, vol.~70, 2017, pp. 1529--1538.

\bibitem{sun2018distributed}
H.~Sun and M.~Hong, ``Distributed non-convex first-order optimization and
  information processing: Lower complexity bounds and rate optimal
  algorithms,'' \emph{arXiv preprint arXiv:1804.02729}, 2018.

\bibitem{sun2019distributed}
------, ``Distributed non-convex first-order optimization and information
  processing: Lower complexity bounds and rate optimal algorithms,'' \emph{IEEE
  Transactions on Signal processing}, vol.~67, no.~22, pp. 5912--5928, 2019.

\bibitem{yi2021linear}
X.~Yi, S.~Zhang, T.~Yang, T.~Chai, and K.~H. Johansson, ``Linear convergence of
  first-and zeroth-order primal--dual algorithms for distributed nonconvex
  optimization,'' \emph{IEEE Transactions on Automatic Control}, vol.~67,
  no.~8, pp. 4194--4201, 2021.

\bibitem{yi2022sublinear}
------, ``Sublinear and linear convergence of modified admm for distributed
  nonconvex optimization,'' \emph{IEEE Transactions on Control of Network
  Systems}, vol.~10, no.~1, pp. 75--86, 2023.

\bibitem{auzinger2011iterative}
W.~Auzinger and J.~Melenk, ``Iterative solution of large linear systems,''
  \emph{Lecture notes, TU Wien}, 2011.

\bibitem{mokhtari2016decentralized}
A.~Mokhtari, W.~Shi, Q.~Ling, and A.~Ribeiro, ``A decentralized second-order
  method with exact linear convergence rate for consensus optimization,''
  \emph{IEEE Transactions on Signal and Information Processing over Networks},
  vol.~2, no.~4, pp. 507--522, 2016.

\bibitem{Wu2020ASP}
X.~Wu, Z.~Qu, and J.~Lu, ``A second-order proximal algorithm for consensus
  optimization,'' \emph{IEEE Transactions on Automatic Control}, vol.~66, pp.
  1864--1871, 2020.

\bibitem{scaman2017optimal}
K.~Scaman, F.~Bach, S.~Bubeck, Y.~T. Lee, and L.~Massouli{\'e}, ``Optimal
  algorithms for smooth and strongly convex distributed optimization in
  networks,'' in \emph{Proceedings of the 34th International Conference on
  Machine Learning}, vol.~70, 2017, pp. 3027--3036.

\bibitem{xu2020accelerated}
J.~Xu, Y.~Tian, Y.~Sun, and G.~Scutari, ``Accelerated primal-dual algorithms
  for distributed smooth convex optimization over networks,'' in
  \emph{Proceedings of the Twenty Third International Conference on Artificial
  Intelligence and Statistics}, vol. 108, 2020, pp. 2381--2391.

\bibitem{antoniadis2011penalized}
A.~Antoniadis, I.~Gijbels, and M.~Nikolova, ``Penalized likelihood regression
  for generalized linear models with non-quadratic penalties.'' \emph{Annals of
  the Institute of Statistical Mathematics}, vol.~63, no.~3, 2011.

\end{thebibliography}

\end{document}